\documentclass{article}
\usepackage[scale=.7]{geometry}
\usepackage{amsmath}
\usepackage{amssymb}
\usepackage{amsthm}
\usepackage{mleftright}
\mleftright
\usepackage{tikz}
\usepackage[backref,hidelinks]{hyperref}
\allowdisplaybreaks

\newcommand{\CC}{\mathbb{C}}
\newcommand{\FF}{\mathbb{F}}
\newcommand{\NN}{\mathbb{N}}
\newcommand{\PP}{\mathbb{P}}
\newcommand{\QQ}{\mathbb{Q}}

\newcommand{\cO}{\mathcal{O}}
\newcommand{\ol}{\overline}
\newcommand{\wt}{\widetilde}
\newcommand{\wE}{\wt{E}}
\newcommand{\wF}{\wt{F}}

\newcommand{\Zp}{Z_2}
\newcommand{\Zm}{Z_{-2}}
\newcommand{\wZp}{\wt{Z}_2}
\newcommand{\wZm}{\wt{Z}_{-2}}

\DeclareMathOperator{\Pic}{Pic}
\DeclareMathOperator{\lct}{lct}
\DeclareMathOperator{\Supp}{Supp}

\newtheorem{theorem}{Theorem}[section]
\newtheorem{conjecture}[theorem]{Conjecture}
\newtheorem{proposition}[theorem]{Proposition}
\newtheorem{lemma}[theorem]{Lemma}

\theoremstyle{definition}
\newtheorem{definition}[theorem]{Definition}

\title{A counterexample to Tian’s Stabilization Conjecture\thanks{This is part of the author's PhD work. The author is grateful to his advisor Y.A. Rubinstein for suggesting this problem, and for numerous discussions and guidance. The research is supported by NSF grants DMS-1906370, 2204347, BSF grant 2020329, and Ann G. Wylie Dissertation Fellowship in 2023-24.}}
\author{Chenzi Jin}
\date{\today}

\makeatletter
\def\thanks#1{\protected@xdef\@thanks{\@thanks\protect\footnotetext{#1}}}
\makeatother

\begin{document}

\maketitle

\begin{abstract}
It was conjectured by Tian that the global log canonical threshold (known as the $\alpha$-invariant) is equal to the level $k$ log canonical threshold (known as the $\alpha_k$-invariant) for all sufficiently large $k$. A weaker folklore conjecture has been that the invariants $\alpha_k$ are eventually monotone. We provide a counterexample to both conjectures.
\end{abstract}

\section{Introduction}

The purpose of this article is to provide a counterexample to Tian's stabilization conjecture on $\alpha$-invariants.

The global log canonical threshold, or $\alpha$-invariant (Definition \ref{def}), was introduced by Tian to prove existence of K\"ahler--Einstein metrics \cite{Tian87}. It was also independently studied in the algebraic geometry literature by Shokurov \cite{Sho}. While Tian's definition was analytic and Shokurov's definition was algebraic, it was shown by Demailly that they coincide \cite{CS}. The quantized version, namely the $\alpha_k$-invariant (Definition \ref{def}), was also introduced by Tian \cite{Tian88,Tian90}, with the observation that $\liminf_{k\to\infty}\alpha_k=\alpha$. By results of Demailly \cite{CS} and Shi \cite{Shi}, $\alpha=\inf_k\alpha_k=\lim_{k\to\infty}\alpha_k$. In 1988 Tian proposed the following conjecture (refined in 2012), that $\alpha_k$ stabilizes for sufficiently large $k$ \cite[Question 1]{Tian90}, \cite[Conjecture 5.3]{Tian12}.
\begin{conjecture}\label{main conjecture}
    For any ample line bundle $L$ over a projective manifold $X$, there exists $k_0\in\NN_+$ such that for $k\geq k_0$, $\alpha_k(L)=\alpha(L)$.
\end{conjecture}
There is also a folklore conjecture, weaker than Conjecture \ref{main conjecture}, stating that the sequence of $\alpha_k$-invariants is eventually monotone (see \cite[\S1.3]{JR} for a discussion and references).
\begin{conjecture}\label{main conjecture weak}
    For any ample line bundle $L$ over a projective manifold $X$, there exists $k_0\in\NN_+$ such that for $k_0\leq k_1\leq k_2$, $\alpha_{k_1}(L)\geq\alpha_{k_2}(L)$.
\end{conjecture}

Conjectures \ref{main conjecture} and \ref{main conjecture weak} were original stated for Fano $X$ and $L=-K_X$. That remains the most interesting case, due to the relation to K\"ahler--Eintein metrics. Some positive results are known for this case. The case of smooth del Pezzo surfaces can be deduced from \cite{Che}. In general, it was shown by Birkar that $\alpha_k(-K_X)=\alpha(-K_X)$ if $k$ is sufficiently large and \textit{divisible}, provided that $\alpha(-K_X)\leq1$ (even for $-K_X$ nef and big) \cite{Bir}. Note that this result does not provide an answer to Conjecture \ref{main conjecture}. When the manifold is toric, Conjecture \ref{main conjecture} was very recently confirmed by the author and Rubinstein for any ample line bundle who also disproved stabilization for other quantized invariants of the form $\alpha_{k,m}$ in the toric setting \cite[Theorems 1.4, 1.6]{JR}.

There is an even stronger version of Conjecture \ref{main conjecture} \cite[Conjecture 5.4]{Tian12}, predicting how large $k_0$ is.

\begin{conjecture}\label{main conjecture strong}
    For any ample line bundle $L$ over a projective manifold $X$, if the section ring $R(X,L):=\oplus_{k=0}^\infty H^0(X,kL)$ is generated by $\oplus_{k=0}^{k_0}H^0(X,kL)$, then for $k\geq k_0$, $\alpha_k(L)=\alpha(L)$. In particular, if the section ring $R(X,L)$ is generated by $H^0(X,L)$, then for any $k\in\NN_+$, $\alpha_k(L)=\alpha(L)$.
\end{conjecture}

A counterexample to Conjecture \ref{main conjecture strong} was provided by Ahmadinezhad--Cheltsov--Schicho \cite{ACS}. They found a smooth surface $S$ with a very ample line bundle $L$ such that the section ring $R(S,L)$ is generated by $H^0(S,L)$, but $\alpha_1(L)>\alpha(L)$.

The main result of this article is a counterexample to Conjectures \ref{main conjecture} and \ref{main conjecture weak}.

\begin{theorem}\label{main theorem}
    There is an ample line bundle $L$ over a smooth surface $S$ such that $\alpha_k$ does not stabilize, nor is it eventually monotone.
\end{theorem}

Note that in this counterexample $-K_S$ is nef and big, although we are not taking $L=-K_S$. It still remains open whether for Fano $X$ and $L=-K_X$ such an example exists. Further counterexamples can be constructed if the positivity is relaxed and we explore these in \cite{JR4}.

This article can be considered as a sequel to the program initiated with Rubinstein \cite{JR,JR2,JRT} to tackle problems on stabilization of quantized algebraic invariants arising in K-stability, and in the absence of such stabilization to consider the refined problem of determining the large $k$ asymptotics of such invariants. In this direction, Theorem \ref{surface main} below also shows that the asymptotics $\alpha_k(L)=\alpha(L)+O(k^{-1})$ \cite[Corollary 5.2]{BJ} \cite[Theorem 1.5]{JRT} is optimal in the setting that either $L$ is ample or $\alpha(L)=\alpha_\ell(L)$ for some $\ell$. As far as we are aware, this is also the first time the sequence $\{\alpha_k\}$ is explicitly computed in any example where the sequence is non-constant.

\bigskip
\noindent\textbf{Organization.}
In \S\ref{setup} we explain the construction of the pair $(S,L)$. In \S\ref{Pic} we compute the intersection numbers necessary for the computation of $\alpha_k$-invariants. As an application we show ampleness of $L$ (Theorem \ref{ample}). In \S\ref{T-var} we make use of the T-variety structure and reduce the computation to a linear programming problem. In \S\ref{computation} we finish the computation of $\alpha_k$-invariants and prove Theorem \ref{main theorem}.

\section{Preliminaries}

\subsection{Log canonical thresholds and \texorpdfstring{$\alpha$}{alpha}-invariants}

\begin{definition}\label{lct def}
    Let $D$ be an effective divisor on a complex manifold $X$ given by local defining functions $\{U_i,f_i\}_{i=1}^N$. Its \textit{log canonical threshold} is
    $$
        \lct\left(D\right):=\sup\left\{c\geq0\,\middle|\,\left|f_i\right|^{-2c}\in L^1_\mathrm{loc}\left(U_i\right)\right\}.
    $$
\end{definition}
\begin{proposition}\label{lct homo}
    For any effective divisor $D$ and $k\in\NN_+$, $\lct(kD)=\frac{1}{k}\lct(D)$.
\end{proposition}
\begin{proof}
    If $f$ is a local defining function of $D$, then $f^k$ is a local defining function of $kD$. The result follows from Definition \ref{lct def}.
\end{proof}

Using this homogeneity we can extend the notion of log canonical threshold to effective $\QQ$-divisors.

\begin{definition}
    Let $D$ be an effective $\QQ$-divisor and suppose $kD$ is a divisor for some $k\in\NN_+$. Then
    $$
        \lct\left(D\right):=k\lct\left(kD\right).
    $$
\end{definition}

\begin{definition}\label{snc def}
    A $\QQ$-divisor $D$ on a complex manifold $X$ is \textit{normal crossing} if for any $p\in X$ there is a coordinate chart on some neighborhood $U$ of $p$ such that $\Supp D\cap U$ is a union of coordinate hyperplanes, i.e., it is the vanishing locus of $z_1\cdots z_\ell$, for some $0\leq\ell\leq n$.
\end{definition}
\begin{lemma}\label{snc lem}
    Consider the function $f(z_1,\ldots,z_n)=\prod_{i=1}^nz_i^{a_i}$, where $a_i\in\NN$. Then for $c\geq0$, $|f|^{-2c}$ is integrable around $0$ if and only if $c<(\max_{1\leq i\leq n}a_i)^{-1}$.
\end{lemma}
\begin{proof}
    Let $D$ denote the unit disc. We compute
    $$
        \int_{D^n}\left|f\left(z\right)\right|^{-2c}dz=\prod_{i=1}^n\int_D\left|z_i\right|^{-2a_ic}dz_i.
    $$
    Therefore $|f|^{-2c}$ is integrable around $0$ if and only if for each $i$, $c<a_i^{-1}$.
\end{proof}

\begin{proposition}\label{nc lct}
    Suppose $D=\sum_{i=1}^\ell a_iD_i$ is a normal crossing $\QQ$-divisor, where $a_i>0$. Then $\lct(D)=(\max_{1\leq i\leq\ell}a_i)^{-1}$.
\end{proposition}
\begin{proof}
    By homogeneity (Proposition \ref{lct homo}) we may assume $D$ is a divisor, i.e., $a_i\in\NN_+$. Recall Definition \ref{snc def}. The result follows from Lemma \ref{snc lem}.
\end{proof}

\begin{definition}\label{def}
    Let $L$ be an ample line bundle. For $k\in\NN_+$,
    \begin{equation}\label{alpha_k def}
        \alpha_k\left(L\right):=\inf_{D\sim kL}k\lct\left(D\right),
    \end{equation}
    and
    \begin{equation}\label{alpha def}
        \alpha\left(L\right):=\inf_{k\in\NN_+}\alpha_k.
    \end{equation}
\end{definition}
In other words, $\alpha(L)=\inf_{D\sim_\QQ L}\lct(D)$.

\subsection{T-varieties}

\begin{definition}
    A \textit{T-variety of complexity $k$} is a normal variety $X$ admitting an effective torus action of codimension $k$.
\end{definition}

To compute $\alpha$-invariants on T-varieties it suffices to look at torus-invariant divisors \cite[Proposition 2.6]{Sus}.
\begin{theorem}\label{T-variety}
    Let $X$ be a T-variety. For any effective divisor $D$, there is a torus-invariant divisor $D'\sim D$ with $\lct(D')\leq\lct(D)$.
\end{theorem}

\section{Construction of the surface and polarization}\label{setup}

In this section we explain the construction of the pair $(S,L)$ and fix some notations. See Figure \ref{figure} for an illustration.

\begin{figure}[ht]
    \centering
    \begin{tikzpicture}
        \draw(-1,2)node[left]{$\Zm$}rectangle(1,-2);
        \draw(1,2)node[right]{$\Zp$};
        \draw(0,-2)node[below]{$\FF_2$};
        \filldraw(1,-1)circle(1pt)node[right]{$p_1$};
        \filldraw(1,0)circle(1pt)node[right]{$p_2$};
        \filldraw(1,.5)circle(1pt)node[right]{$p_3$};
        \filldraw(1,1)circle(1pt)node[right]{$p_4$};
        \draw[->](1.8,0)--(2.5,0)node[midway,above]{$\pi$};
        \draw(3,-2)node[below]{$\PP^1$}--(3,2);
        \filldraw(3,-1)circle(1pt)node[right]{$\pi(p_1)$};
        \filldraw(3,0)circle(1pt)node[right]{$\pi(p_2)$};
        \filldraw(3,.5)circle(1pt)node[right]{$\pi(p_3)$};
        \filldraw(3,1)circle(1pt)node[right]{$\pi(p_4)$};
        \draw[->](-2.2,0)--(-1.5,0);
        \draw(-5.3,2)rectangle(-3.3,-2);
        \draw(-4.3,-2)node[below]{$\ol{S}$};
        \draw(-5.5,-1.1)node[left]{$\ol{F}_1$}--(-4.1,-.4);
        \draw(-3.1,-1.1)node[right]{$\ol{E}_1$}--(-4.5,-.4);
        \draw(-5.5,-.1)node[left]{$\ol{F}_2$}--(-4.1,.6);
        \draw(-3.1,-.1)node[right]{$\ol{E}_2$}--(-4.5,.6);
        \draw(-5.5,.4)node[left]{$\ol{F}_3$}--(-4.1,1.1);
        \draw(-3.1,.4)node[right]{$\ol{E}_3$}--(-4.5,1.1);
        \draw(-5.5,.9)node[left]{$\ol{F}_4$}--(-4.1,1.6);
        \draw(-3.1,.9)node[right]{$\ol{E}_4$}--(-4.5,1.6);
        \draw[->](-7.1,0)--(-6.4,0);
        \draw(-10.2,2)node[left]{$\wZm$}rectangle(-8.2,-2);
        \draw(-8.2,2)node[right]{$\wZp$};
        \draw(-9.2,-2)node[below]{$S$};
        \draw(-10.4,-1.2)node[left]{$\wF_1$}--(-9.5,-.3);
        \draw(-8,-1.2)node[right]{$\wE_1$}--(-8.9,-.3);
        \draw(-9.9,-.5)--(-8.5,-.5)node[midway,below]{$E_1$};
        \draw(-10.4,-.2)node[left]{$\wF_2$}--(-9.5,.7);
        \draw(-8,-.2)node[right]{$\wE_2$}--(-8.9,.7);
        \draw(-9.9,.5)--(-8.5,.5)node[midway,below]{$E_2$};
        \draw(-10.4,.4)node[left]{$\wF_3$}--(-9,1.1);
        \draw(-8,.4)node[right]{$\wE_3$}--(-9.4,1.1);
        \draw(-10.4,.9)node[left]{$\wF_4$}--(-9,1.6);
        \draw(-8,.9)node[right]{$\wE_4$}--(-9.4,1.6);
    \end{tikzpicture}
    \caption{The construction of the surface $S$.}\label{figure}
\end{figure}
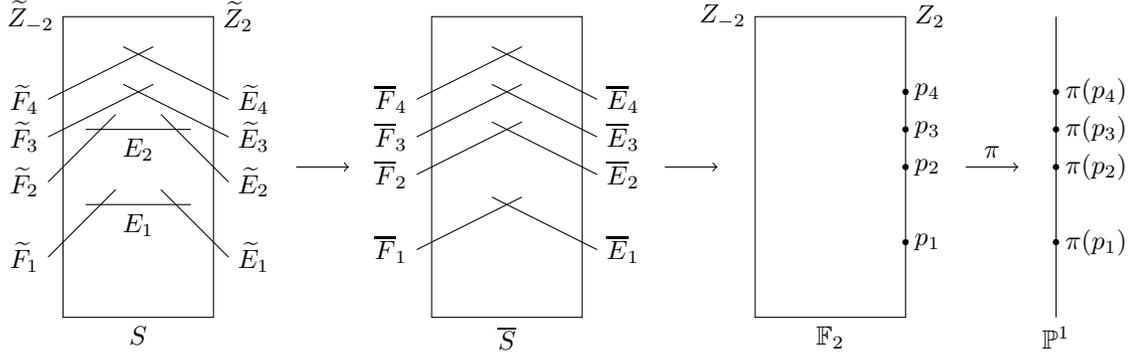

Recall that the Hirzebruch surface $\FF_2$ is a $\PP^1$-bundle over $\PP^1$. Let $\pi:\FF_2\to\PP^1$ denote this projection map. We also fix two sections of this bundle, namely the $-2$-curve $\Zm$, and a smooth $2$-curve $\Zp$.

Pick four points $p_1,\ldots,p_4\in\Zp$. Let $F_i$ denote the fiber $\pi^{-1}(\pi(p_i))$, and $F$ a general fiber on $\FF_2$, i.e., $F=\pi^{-1}(p)$ for some $p\in U$, where
$$
    U:=\PP^1\setminus\left\{\pi\left(p_1\right),\ldots,\pi\left(p_4\right)\right\}.
$$
Let $\ol{S}$ denote the blow-up of $\FF_2$ at the four points $p_1,\ldots,p_4$. Let $\ol{E}_i$ denote the corresponding exceptional divisor, and $\ol{F}_i$ the proper transform of $F_i$.

Finally, let $S$ denote the blow-up of $\ol{S}$ at the two points $\ol{E}_1\cap\ol{F}_1$ and $\ol{E}_2\cap\ol{F}_2$. Let $E_1$ and $E_2$ denote the corresponding exceptional divisor, and $\wE_i$ (resp.\ $\wF_i$) the proper transform of $\ol{E}_i$ (resp.\ $\ol{F}_i$).

Let $f:S\to\FF_2$ denote the composition of blow-up maps, $\wZp$ (resp.\ $\wZm$, $\wF$) the proper transform of $\Zp$ (resp.\ $\Zm$, $F$). We consider the line bundle
\begin{equation}\label{L def}
    L:=\cO_S\left(2\wZp+2\wZm+3\wF+\wE_1+\wF_1+E_1+\wE_2+\wF_2+E_2\right).
\end{equation}

\section{Picard group and intersection numbers}\label{Pic}

We compute the intersection numbers on $\Pic(S)$ in this section. As an application, we show that $L$ is ample (Theorem \ref{ample}).

Recall that $\Pic(\FF_2)$ is a free abelian group of rank $2$, generated by $[\Zp]$ and $[F]$, with $2$-by-$2$ intersection matrix
$$
    \begin{pmatrix}
        2&1\\
        1&0
    \end{pmatrix},
$$
and
$$
    \Zm\sim\Zp-2F.
$$
After the six blow-ups, $\Pic(S)$ is a free abelian group of rank $8$, generated by $[\wZp]$, $[\wF]$, and the exceptional divisors $[\wE_1],\ldots,[\wE_4],[E_1],[E_2]$, with $8$-by-$8$ intersection matrix (off-diagonal zero entries are omitted)
\begin{equation}\label{intersection matrix}
    \begin{pmatrix}
        -2&1& 1&1 & 1& 1\\
         1&0\\
         1& &-2&  &  &  & 1\\
         1& &  &-2&  &  &  &1\\
         1& &  &  &-1\\
         1& &  &  &  &-1\\
          & & 1&  &  &  &-1\\
          & &  & 1&  &  &  &-1\\
    \end{pmatrix},
\end{equation}
and
\begin{equation}\label{linear equivalences}
    \begin{gathered}
        \wF_1\sim\wF-\wE_1-2E_1,\qquad\wF_2\sim\wF-\wE_2-2E_2,\qquad\wF_3\sim\wF-\wE_3,\qquad\wF_4\sim\wF-\wE_4,\\
        \wZm\sim\wZp-2\wF+\wE_1+\wE_2+\wE_3+\wE_4+E_1+E_2.
    \end{gathered}
\end{equation}
See also Figure \ref{negative curves}.

\begin{figure}[ht]
    \centering
    \begin{tikzpicture}
        \draw[dotted](-10.2,2)rectangle(-8.2,-2);
        \draw[ultra thick](-10.4,-1.2)node[left]{$\wF_1$}--(-9.5,-.3);
        \draw[ultra thick](-8,-1.2)node[right]{$\wE_1$}--(-8.9,-.3);
        \draw[ultra thin](-9.9,-.5)--(-8.5,-.5)node[midway,below]{$E_1$};
        \draw[ultra thick](-10.4,-.2)node[left]{$\wF_2$}--(-9.5,.7);
        \draw[ultra thick](-8,-.2)node[right]{$\wE_2$}--(-8.9,.7);
        \draw[ultra thin](-9.9,.5)--(-8.5,.5)node[midway,below]{$E_2$};
        \draw[ultra thin](-10.4,.4)node[left]{$\wF_3$}--(-9,1.1);
        \draw[ultra thin](-8,.4)node[right]{$\wE_3$}--(-9.4,1.1);
        \draw[ultra thin](-10.4,.9)node[left]{$\wF_4$}--(-9,1.6);
        \draw[ultra thin](-8,.9)node[right]{$\wE_4$}--(-9.4,1.6);
        \draw[ultra thick](-10.2,-2)--(-10.2,2)node[left]{$\wZm$};
        \draw[ultra thick](-8.2,-2)--(-8.2,2)node[right]{$\wZp$};
    \end{tikzpicture}
    \caption{Some curves with negative self-intersection on $S$. The thick curves have self-intersection $-2$, and the thin curves have self-intersection $-1$.}\label{negative curves}
\end{figure}
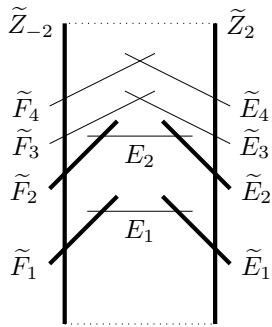

Recall \eqref{L def}. We compute the following intersection numbers.
\begin{equation}\label{L intersection}\begin{gathered}
    L\cdot\wZp=1,\qquad L\cdot\wZm=1,\qquad L\cdot\wF=4,\\
    L\cdot\wE_1=1,\qquad L\cdot\wF_1=1,\qquad L\cdot E_1=1,\\
    L\cdot\wE_2=1,\qquad L\cdot\wF_2=1,\qquad L\cdot E_2=1,\\
    L\cdot\wE_3=2,\qquad L\cdot\wF_3=2,\qquad L\cdot\wE_4=2,\qquad L\cdot\wF_4=2.
\end{gathered}\end{equation}

To prove ampleness we use Nakai--Moishezon criterion (see, e.g., \cite[Ch. V, Theorem 1.10]{Har}).

\begin{theorem}
    A divisor $D$ on the surface $X$ is ample if and only if $D^2>0$ and $D\cdot C>0$ for all irreducible curves $C$ in $X$.
\end{theorem}

\begin{theorem}\label{ample}
    The line bundle $L$ defined by \eqref{L def} is ample.
\end{theorem}
\begin{proof}
    Since $L$ is effective, it suffices to prove $L\cdot C>0$ for all irreducible curves $C$.

    Let
    $$
        D:=2\wZp+2\wZm+3\wF+\wE_1+\wF_1+E_1+\wE_2+\wF_2+E_2\sim L.
    $$
    If $C$ lies in $\Supp D$, it has been shown in \eqref{L intersection} that $L\cdot C>0$. If $C$ does not lie in $\Supp D$, it suffices to show that $C$ intersects $\Supp D$. Recall the blow-down map $f:S\to\FF_2$ and the projection $\pi:\FF_2\to\PP^1$. The image $(\pi\circ f)(C)$ is either $0$-dimensional (i.e., a point $p\in U\subseteq\PP^1$) or $1$-dimensional (i.e., the whole $\PP^1$). In the former case, $C=(\pi\circ f)^{-1}(p)$, so $C$ intersects $\wZp$ and $\wZm$. In the latter case, $C$ intersects $\wF$. This completes the proof.
\end{proof}

\section{Torus-invariant divisors}\label{T-var}

In this section we discuss the T-variety structure of $S$, and reduce the family of divisors to be considered in computing $\alpha_k$-invariants.

Recall that $\FF_2$ admits a toric variety structure, with $\Zp$ fixed by the torus action. We regard it as a T-variety of complexity $1$ by restricting the torus action to the one dimensional subgroup that leaves each fiber invariant. More specifically, recall that $\FF_2\setminus\Zm$ is the total space of the line bundle $\cO_{\PP^1}(2)$. The $\CC^*$-action is the scalar multiplication on each fiber of this line bundle.

This action can then be lifted to $S$, equipping $S$ with the structure of T-variety of complexity $1$.

\begin{figure}[ht]
    \centering
    \begin{tikzpicture}
        \draw[gray](-10.2,2)rectangle(-8.2,-2);
        \draw[gray](-10.4,-1.2)--(-9.5,-.3);
        \draw[gray](-8,-1.2)--(-8.9,-.3);
        \draw[gray](-9.9,-.5)--(-8.5,-.5);
        \filldraw(-9.7,-.5)circle(1pt);
        \filldraw(-8.7,-.5)circle(1pt);
        \draw[gray](-10.4,-.2)--(-9.5,.7);
        \draw[gray](-8,-.2)--(-8.9,.7);
        \draw[gray](-9.9,.5)--(-8.5,.5);
        \filldraw(-9.7,.5)circle(1pt);
        \filldraw(-8.7,.5)circle(1pt);
        \draw[gray](-10.4,.4)--(-9,1.1);
        \draw[gray](-8,.4)--(-9.4,1.1);
        \filldraw(-9.2,1)circle(1pt);
        \draw[gray](-10.4,.9)--(-9,1.6);
        \draw[gray](-8,.9)--(-9.4,1.6);
        \filldraw(-9.2,1.5)circle(1pt);
        \draw[ultra thick](-10.2,-2)--(-10.2,2);
        \draw[ultra thick](-8.2,-2)--(-8.2,2);
    \end{tikzpicture}
    \caption{The fixed points under the torus action.}\label{fixed points}
\end{figure}
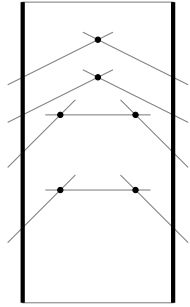

\begin{lemma}\label{T divisor}
    The torus-invariant curves on $S$ are $\wZp,\wZm,\wE_1,\ldots,\wE_4,\wF_1,\ldots,\wF_4,E_1,E_2$, and $(\pi\circ f)^{-1}(p)$ for $p\in U$ (recall \S\ref{setup}).
\end{lemma}
\begin{proof}
    Notice that the set of fixed points under the torus action consists of $\wZp\cup\wZm$ and six points $E_1\cap\wE_1,E_1\cap\wF_1,E_2\cap\wE_2,E_2\cap\wF_2,\wE_3\cap\wF_3,\wE_4\cap\wF_4$ (see Figure \ref{fixed points}). Let $C$ be a torus-invariant curve. If $C$ is fixed by the torus action, then $C=\wZp$ or $\wZm$. If $C$ is not fixed by the torus action, then $C$ contains a one-dimensional orbit. In that case, $C=\wE_1,\ldots,\wE_4,\wF_1,\ldots,\wF_4,E_1,E_2$, or $(\pi\circ f)^{-1}(p)$ for $p\in U$.
\end{proof}

\begin{lemma}\label{alpha_k simp}
    For $k\in\NN_+$,
    $$
        \alpha_k\left(L\right)=\inf_Dk\lct\left(D\right),
    $$
    where the infimum is taken over divisors of the form
    $$
        D=a_1\wZp+a_2\wZm+a_3\wE_1+\cdots+a_6\wE_4+a_7\wF_1+\cdots+a_{10}\wF_4+a_{11}E_1+a_{12}E_2\sim kL.
    $$
\end{lemma}
\begin{proof}
    By Theorem \ref{T-variety} and Lemma \ref{T divisor}, it suffices to look at divisors of the form
    $$
        D=a_1\wZp+a_2\wZm+a_3\wE_1+\cdots+a_6\wE_4+a_7\wF_1+\cdots+a_{10}\wF_4+a_{11}E_1+a_{12}E_2+\sum_{p\in U}a_p\left(\pi\circ f\right)^{-1}\left(p\right).
    $$
    We may drop the last term. Indeed, by \eqref{linear equivalences},
    \begin{multline*}
        D\sim D':=a_1\wZp+a_2\wZm+a_3\wE_1+a_4\wE_2+a_5\wE_3+\left(a_6+\sum_{p\in U}a_p\right)\wE_4\\
        +a_7\wF_1+a_8\wF_2+a_9\wF_3+\left(a_{10}+\sum_{p\in U}a_p\right)\wF_4+a_{11}E_1+a_{12}E_2.
    \end{multline*}
    By Proposition \ref{nc lct}, $\lct(D')\leq\lct(D)$. This completes the proof.
\end{proof}

\section{The computation of \texorpdfstring{$\alpha_k$}{alphak}-invariants}\label{computation}

In this section we finish the computation of $\alpha_k$-invariants, providing an explicit counterexample to Tian's stabilization conjecture.

\begin{theorem}\label{surface main}
    Let $(S,L)$ be as in \S\ref{setup}. Then
    $$
        \alpha_k\left(L\right)=\begin{cases}
            \dfrac{1}{8},&k\text{ is even};\\
            \dfrac{k}{8k-1},&k\text{ is odd},
        \end{cases}
    $$
    and
    $$
        \alpha\left(L\right)=\frac{1}{8}.
    $$
    In particular, $\alpha_k(L)\geq\alpha(L)$ with equality if and only if $k$ is even, and $\alpha_k(L)$ is not eventually monotone.
\end{theorem}
\begin{proof}
    Let
    $$
        D=a_1\wZp+a_2\wZm+a_3\wE_1+\cdots+a_6\wE_4+a_7\wF_1+\cdots+a_{10}\wF_4+a_{11}E_1+a_{12}E_2\sim kL,
    $$
    where $a_1,\ldots,a_{12}\in\NN_+$. By \eqref{intersection matrix} and \eqref{linear equivalences} (see also Figure \ref{negative curves}),
    \begin{align*}
        &D\cdot\wF=a_1+a_2,\\
        &D\cdot\left(-\wE_2+\wE_3\right)=2a_4-a_5+a_9-a_{12},\\
        &D\cdot\left(-\wE_1+\wE_3\right)=2a_3-a_5+a_9-a_{11},\\
        &D\cdot\left(\wZp+\wE_3+\wE_4\right)=a_3+a_4+a_9+a_{10},\\
        &D\cdot\left(\wZm+\wF_3+\wE_4\right)=a_5+a_6+a_7+a_8,\\
        &D\cdot\left(2\wZp+\wE_1+\wE_2+2\wE_3\right)=2a_6+2a_9+a_{11}+a_{12},\\
        &D\cdot\left(2\wZp+\wE_1+\wE_2+2\wE_4\right)=2a_5+2a_{10}+a_{11}+a_{12}.\\
    \end{align*}
    Since $D\sim kL$, combining with \eqref{L intersection} we get
    \begin{align}
        &a_1+a_2=4k,\label{12}\\
        &2a_4-a_5+a_9-a_{12}=k,\label{45912}\\
        &2a_3-a_5+a_9-a_{11}=k,\label{35911}\\
        &a_3+a_4+a_9+a_{10}=5k,\label{34910}\\
        &a_5+a_6+a_7+a_8=5k,\label{5678}\\
        &2a_6+2a_9+a_{11}+a_{12}=8k,\label{69}\\
        &2a_5+2a_{10}+a_{11}+a_{12}=8k.\label{510}
    \end{align}
    By \eqref{12},
    \begin{equation}\label{ub12}
        a_1,a_2\leq4k.
    \end{equation}
    By \eqref{34910} and \eqref{5678}, \begin{equation}\label{ub345678910}
        a_3,\ldots,a_{10}\leq 5k.
    \end{equation}
    By \eqref{69} or \eqref{510},
    \begin{equation}\label{ub1112}
        a_{11},a_{12}\leq8k.
    \end{equation}

    Combining \eqref{ub12}, \eqref{ub345678910}, and \eqref{ub1112}, by Proposition \ref{nc lct}, 
    $$
        \lct\left(D\right)\geq\left(8k\right)^{-1}.
    $$
    Since $D$ is arbitrary, by Lemma \ref{alpha_k simp},
    \begin{equation}\label{even bound}
        \alpha_k\left(L\right)\geq\frac{1}{8}.
    \end{equation}
    If $k$ is even, consider the divisor
    $$
        D':=2k\wZp+2k\wZm+\frac{9k}{2}\wE_1+\frac{9k}{2}\wF_1+8kE_1+\frac{k}{2}\wE_2+\frac{k}{2}\wF_2.
    $$
    Using \eqref{linear equivalences} we can check that $D'\sim kL$. By Proposition \ref{nc lct},
    $$
        \lct\left(D'\right)=\left(8k\right)^{-1}.
    $$
    By Lemma \ref{alpha_k simp} and \eqref{even bound},
    $$
        \alpha_k\left(L\right)=\frac{1}{8}.
    $$

    Finally, let $k$ be odd. We claim that $a_{11},a_{12}\neq8k$. Indeed, assume $a_{11}=8k$. By \eqref{69} and \eqref{510}, 
    $$
        a_5=a_6=a_9=a_{10}=a_{12}=0.
    $$
    Hence by \eqref{45912},
    $$
        2a_4=k,
    $$
    contradicting the fact that $k$ is odd. Similarly, assume $a_{12}=8k$. By \eqref{69} and \eqref{510},
    $$
        a_5=a_6=a_9=a_{10}=a_{11}=0.
    $$
    Hence by \eqref{35911},
    $$
        2a_3=k,
    $$
    contradicting the fact that $k$ is odd. Therefore, \eqref{ub1112} can be strengthened to
    $$
        a_{11},a_{12}\leq8k-1.
    $$
    Combining with \eqref{ub12} and \eqref{ub345678910}, by Proposition \ref{nc lct},
    $$
        \lct\left(D\right)\geq\left(8k-1\right)^{-1}.
    $$
    Since $D$ is arbitrary, by Lemma \ref{alpha_k simp},
    \begin{equation}\label{odd bound}
        \alpha_k\left(L\right)\geq\frac{k}{8k-1}.
    \end{equation}
    On the other hand, consider the divisor
    $$
        D'':=2k\wZp+2k\wZm+\frac{9k-1}{2}\wE_1+\frac{9k-1}{2}\wF_1+\left(8k-1\right)E_1+\frac{k+1}{2}\wE_2+\frac{k+1}{2}\wF_2+E_2.
    $$
    Using \eqref{linear equivalences} we can check that $D''\sim kL$. By Proposition \ref{nc lct},
    $$
        \lct\left(D''\right)=\left(8k-1\right)^{-1}.
    $$
    By Lemma \ref{alpha_k simp} and \eqref{odd bound},
    $$
        \alpha_k\left(L\right)=\frac{k}{8k-1}.
    $$
    This completes the proof.
\end{proof}


\begin{thebibliography}{}

\bibitem{ACS}
H. Ahmadinezhad, I. Cheltsov, J. Schicho, \textit{On a conjecture of Tian}, Math. Z. 288 (2018), 217--241.

\bibitem{Bir}
C. Birkar, \textit{Singularities of linear systems and boundedness of Fano varieties}, Ann. of Math. (2) 193 (2022), 347--405.

\bibitem{BJ}
H. Blum, M. Jonsson, \textit{Thresholds, valuations, and K-stability}, Adv. Math. 365 (2020), 107062.

\bibitem{Che}
I. Cheltsov, \textit{Log canonical thresholds of del Pezzo surfaces}, Geom. Funct. Anal. 11 (2008), 1118--1144.

\bibitem{CS}
I.A. Cheltsov, K.A. Shramov, \textit{Log-canonical thresholds for nonsingular Fano threefolds, with an appendix by J.-P. Demailly},
Russian Math. Surveys 63 (2008), 859--958.

\bibitem{Har}
R. Hartshorne, \textit{Algebraic Geometry}, Graduate Texts in Mathematics 52, Springer, 1977.

\bibitem{JR}
C. Jin, Y.A. Rubinstein, \textit{Tian's stabilization problem for toric Fanos}, preprint, 2024, arxiv:2403.17262.

\bibitem{JR2}
C. Jin, Y.A. Rubinstein, \textit{Asymptotics of quantized barycenters of lattice polytopes with applications to algebraic geometry, (with an appendix by Y. Liu)}, preprint, 2024, arxiv:2406.18969.

\bibitem{JR4}
C. Jin, Y.A. Rubinstein, \textit{Stabilization and non-stabilization of log canonical thresholds}, preprint in preparation, 2024.

\bibitem{JRT}
C. Jin, Y.A. Rubinstein, G. Tian, \textit{Asymptotics of discrete Okounkov bodies and thresholds}, preprint, 2024, arxiv:2410.20694.

\bibitem{Shi}
Y. Shi, \textit{On the $\alpha$-invariants of cubic surfaces with Eckardt points}, Adv. Math. 225 (2010), 1285--1307.

\bibitem{Sho}
V. Shokurov, \textit{Three-dimensional log perestroikas}, Izv. Math. 56 (1992), 105--203.

\bibitem{Sus}
H. S\"uss, \textit{K\"ahler--Einstein metrics on symmetric Fano T-varieties}, Adv. Math. 246 (2013), 100--113.

\bibitem{Tian87}
G. Tian, \textit{On K\"ahler--Einstein metrics on certain K\"ahler manifolds with $c_1(M)>0$}, Invent. Math. 89 (1987), 225--246.

\bibitem{Tian88}
G. Tian, \textit{K\"ahler metrics on algebraic manifolds}, Ph.D. thesis, Harvard University, 1988.

\bibitem{Tian90}
G. Tian, \textit{On a set of polarized K\"ahler metrics on algebraic manifolds}, J. Differential Geom. 32 (1990), 99--130.

\bibitem{Tian12}
G. Tian, \textit{Existence of Einstein metrics on Fano manifolds}, in: Metric and diﬀerential geometry, Progr. Math. 297, Birkh\"auser/Springer, Basel, 2012, 119--159.

\end{thebibliography}
\end{document}